\documentclass[11pt]{article}
\usepackage{amsmath,amssymb,amsthm, latexsym,color,epsfig,enumerate,a4, graphicx}

\date{}
\title{\vspace{-0.9cm}Finding paths in sparse random graphs requires many queries}
\author{
Asaf Ferber \thanks{Department of Mathematics, Yale University, and
Department of Mathematics, MIT, USA. Emails: asaf.ferber@yale.edu,
ferbera@mit.edu.} \and Michael Krivelevich\thanks{School of Mathematical Sciences, Raymond and Beverly Sackler Faculty of Exact Sciences, Tel Aviv University, Tel Aviv, 6997801, Israel. Email: krivelev@post.tau.ac.il. Research supported in part by USA-Israel BSF Grant
2010115 and by grant 912/12 from the Israel Science Foundation.}
 \and Benny Sudakov
\thanks{Department of Mathematics, ETH, 8092 Zurich, Switzerland.
Email: benjamin.sudakov@math.ethz.ch. Research supported in part by
SNSF grant 200021-149111.} \and Pedro Vieira \thanks{Department of
Mathematics, ETH, 8092 Zurich, Switzerland.  Email:
pedro.vieira@math.ethz.ch.} }

\allowdisplaybreaks

\makeatletter
\newtheorem*{rep@theorem}{\rep@title}
\newcommand{\newreptheorem}[2]{%
\newenvironment{rep#1}[1]{%
 \def\rep@title{#2 \ref{##1}}%
 \begin{rep@theorem}}%
 {\end{rep@theorem}}}
\makeatother

\newtheorem{theorem}{Theorem}
\newreptheorem{theorem}{Theorem}

\theoremstyle{plain}
\newtheorem{thm}{Theorem}

\newtheorem{lemma}[thm]{Lemma}

\newtheorem*{claim}{Claim}
\newtheorem*{conj*}{Conjecture}
\newtheorem*{ques}{Question}

\theoremstyle{definition}

\newtheorem*{dfn*}{Definition}

\newcommand{\NN}{\mathbb{N}}
\newcommand{\EE}{\mathbb{E}}

\begin{document}

\maketitle
\begin{abstract}

We discuss a new algorithmic type of problem in random graphs
studying the minimum number of queries one has to ask about adjacency
between pairs of vertices of a random graph $G\sim {\mathcal
G}(n,p)$ in order to find a subgraph which possesses some target
property with high probability. In this paper we focus on finding
long paths in $G\sim \mathcal G(n,p)$ when
$p=\frac{1+\varepsilon}{n}$ for some fixed constant $\varepsilon>0$.
This random graph is known to have typically linearly long paths.

To have $\ell$ edges with high probability in $G\sim \mathcal
G(n,p)$ one clearly needs to query at least
$\Omega\left(\frac{\ell}{p}\right)$ pairs of vertices. Can we find a
path of length $\ell$ economically, i.e., by querying roughly that
many pairs? We argue that this is not possible and one needs to
query significantly more pairs. We prove that any randomised
algorithm which finds a path of length $\ell=\Omega\left(\frac{\log\left(\frac{1}{\varepsilon}\right)}{\varepsilon}\right)$ with at least
constant probability in $G\sim \mathcal G(n,p)$ with $p=\frac{1+\varepsilon}{n}$ must query at least
$\Omega\left(\frac{\ell}{p\varepsilon
\log\left(\frac{1}{\varepsilon}\right)}\right)$ pairs of vertices.
This is tight up to the $\log\left(\frac{1}{\varepsilon}\right)$
factor.

\end{abstract}

\section{Introduction}
Let ${\cal P}$ be a monotone increasing graph property (that is, a property of graphs that cannot be violated
by adding edges). Suppose that the edge probability $p=p(n)$ is chosen so that a random graph $G$ drawn from the probability space $\mathcal G(n,p)$ has ${\cal P}$ with high probability (whp). How many queries of the type ``is $(i,j)\in E(G)$?" are needed for an adaptive algorithm interacting with the probability space  $\mathcal G(n,p)$ in order to reveal whp a subgraph $G'\subseteq G$ possessing $\mathcal P$?

This fairly natural algorithmic setting (see the excellent survey
of Frieze and McDiarmid  \cite{FM97} for an
extensive coverage of a variety of problems and results in
Algorithmic Theory of Random Graphs) has been considered implicitly in several papers on random graphs (e.g. \cite{ks13}, \cite{ABOW08}), but apparently has been stated explicitly only in the companion paper \cite{FKSV1} of the authors. Notice that in this framework the issue of concern is not the amount of computation required for the algorithm to find a target structure, but rather the amount of its interaction with the underlying probability space.

In the discussion below, we assume some basic familiarity with results about the probability space $\mathcal G(n,p)$; the reader is advised to consult monographs \cite{JLR} and \cite{bollobasrandom} for background on the subject.

In general, given a monotone property $\mathcal P$, what can we
expect? If all $n$-vertex graphs belonging to ${\cal P}$ have at
least $m$ edges, then the algorithm should get at least $m$ positive
answers to hit the target property with the required absolute
certainty. This means that the obvious lower bound in this case is
at least $(1+o(1))m/p$ queries. Perhaps one of the simplest graph properties to consider in this respect is connectedness: for any connected graph $G$ on $n$ vertices a spanning tree can be found after $n-1$ queries with positive answers -- the algorithm starts with an arbitrary vertex $v\in V(G)$, and each time queries the pairs leaving the current tree until the first edge is found, the tree is then updated by appending this edge. Thus for the regime where $\mathcal G(n,p)$ is whp connected (which is when $p(n)\ge\frac{\ln n+\omega(n)}{n}$ with $\lim_{n\rightarrow\infty}\omega(n)=1$), we get an algorithm whp discovering a spanning tree after querying $(1+o(1))n/p$ pairs of vertices. 

A much more challenging problem is that of Hamiltonicity, i.e., of finding a Hamilton cycle. In this case the trivial lower bound translates to $n$ positive answers. In  \cite{FKSV1} we show that this lower bound is tight by providing an adaptive algorithm interacting with the probability space
$\mathcal G(n,p)$, which whp finds a Hamilton cycle in $G\sim
\mathcal G(n,p)$ after obtaining only $(1+o(1))n$ positive answers
(provided $p$ is above the sharp threshold for Hamiltonicity in
$\mathcal G(n,p)$). 

Yet another positive example is that of uncovering a giant component in the supercritical regime $p=\frac{1+\varepsilon}{n}$. Though this was not the main concern in \cite{ks13}, the second and the third author presented there a very natural adaptive algorithm (essentially performing the Depth First Search (DFS) on a random input $G\sim \mathcal G(n,p)$), typically discovering a connected component of size at least $\epsilon n/2$ after querying $\epsilon n^2/2$ vertex pairs.

Upon reviewing these results, the reader may arrive at a conclusion that the above stated trivial lower bound for this type of problems is nearly tight for almost every natural graph property. However, this happens {\bf not} to be the case, and the main qualitative goal of the present paper is to provide such a negative example, including its analysis. Here we focus on the property of containing a
path of length $\ell$ in the supercritical regime in $G\sim\mathcal G(n,p)$, that is, when
$p=\frac{1+\varepsilon}{n}$ for some fixed constant $\varepsilon>0$. For this regime, $G\sim\mathcal G(n,p)$ is known to contain whp a path of length linear in $n$, due to the classical result of Ajtai, Koml\'os and Szemer\'edi \cite{AKS81} (see \cite{ks13} for a recent simple proof of this fact.)
Note that in order to have $\ell$ edges with high probability in
$G\sim\mathcal G(n,p)$ one needs to query at least
$\Omega\left(\frac{\ell}{p}\right)$ pairs of vertices. Can we find a
path of length $\ell$ by asking roughly that many queries, as in the
case of Hamiltonicity mentioned above? We show that in this case one
actually needs to query significantly more pairs of vertices:

\begin{reptheorem}{thm:noalgorithm}
There exists an absolute constant $C>0$ such that the following
holds. For every constant $q\in (0,1)$ there exist
$n_0,\varepsilon_0>0$ such that for every fixed $\varepsilon\in
(0,\varepsilon_0)$ and any $n\ge n_0$ there is no adaptive algorithm
which reveals a path of length $\ell\ge
\frac{3C}{\varepsilon}\ln\left(\frac{1}{\varepsilon}\right)$ with
probability at least $q$ in $G\sim \mathcal G\left(n,p\right)$,
where $p=\frac{1+\varepsilon}{n}$, by querying at most
$\frac{q\ell}{8640Cp\varepsilon\ln\left(\frac{1}{\varepsilon}\right)}$
pairs of vertices.
\end{reptheorem}

Notice that \cite{ks13}  presents a simple adaptive DFS
algorithm, finding a path of length
$\frac{1}{5}\varepsilon^2n$ with probability at least
$1-\exp\left(\Omega(\varepsilon n)\right)$ in $G\sim \mathcal
G(n,p)$ after querying only $O\left(\varepsilon n^2\right)$ pairs of
vertices. In fact, if one uses the same algorithm to find a path of
length $\ell\le \frac{1}{5}\varepsilon^2n$ in $G\sim \mathcal
G(n,p)$ then the same argument shows that such a path is found with
probability at least
$1-\exp\left(\Omega\left(\frac{\ell}{\varepsilon}\right)\right)$
after querying at most $O\left(\frac{\ell}{p\varepsilon}\right)$
pairs of vertices. This shows that up to the
$\Theta\left(\log\left(\frac{1}{\varepsilon}\right)\right)$ factor,
Theorem~\ref{thm:noalgorithm} is tight.

The key ingredient of the proof of Theorem~\ref{thm:noalgorithm} is the following result of independent interest.

\begin{reptheorem}{thm:notmanyvertexdisjointpaths}
There exist constants $C,\varepsilon_0>0$ such that for every fixed
$\varepsilon\in(0, \varepsilon_0)$ and $p=\frac{1+\varepsilon}{n}$
we have whp that a graph $G\sim \mathcal G\left(n,p\right)$ does not
contain a set of vertex disjoint paths of lengths at least
$\frac{C}{\varepsilon}\ln \left(\frac{1}{\varepsilon}\right)$ whose
union covers at least $13\varepsilon^2n$ vertices.
\end{reptheorem}

The rest of this paper is organised as follows. In Section~\ref{sec:auxiliarylemmas} we provide auxiliary lemmas needed for the proofs of Theorem~\ref{thm:noalgorithm} and \ref{thm:notmanyvertexdisjointpaths}. In Section~\ref{sec:proofthm1} we prove Theorem~\ref{thm:noalgorithm} assuming Theorem~\ref{thm:notmanyvertexdisjointpaths}. In Section~\ref{sec:proofthm2} we prove Theorem~\ref{thm:notmanyvertexdisjointpaths}. Finally, in Section~\ref{sec:concludingremarks} we discuss some concluding remarks.

\vspace{3mm}
\textbf{Notation.}
Our notation is fairly standard. Given a natural number $n$ we use $[n]$ to denote the set $\{1,2,\ldots,n\}$. Moreover, given a set $V$ we use $S_{V}$ to denote the permutation group of $V$ and $\binom{V}{2}$ to denote the set of all (unordered) pairs of elements in $V$.

Given a subset $S$ of the vertex set of a graph $G$, $G[S]$ denotes the subgraph of $G$ induced by the vertices in $S$, i.e. the graph with vertex set $S$ whose edges are the ones of $G$ between vertices in $S$.

A subgraph $P$ of the graph $G$ is called a \textit{path} if $V(P)=\{v_1,\ldots,v_{\ell}\}$ and the edges of $P$ are $v_1v_2$, $v_2v_3$, $\ldots$, $v_{\ell-1}v_{\ell}$. We shall oftentimes refer to $P$ simply by $v_1v_2\ldots v_{\ell}$. We say that such a path $P$ has \textit{length} $\ell-1$ (number of edges) and \textit{size} $\ell$ (number of vertices).

If $G$ is a graph then the \emph{$2$-core} of $G$ is the maximal induced subgraph of $G$ of minimum degree at least $2$. If no such subgraph exists then the $2$-core of $G$ is the empty graph.

Given an ordered set $V$ and a real number $p \in [0,1]$, the
binomial random graph model $\mathcal G(V,p)$ is a probability space whose
ground set consists of all labeled graphs on the vertex set $V$. We
can describe the probability distribution of $G\sim \mathcal G(V,p)$
by saying that each pair of elements of $V$ forms an edge in $G$
independently with probability $p$. If $V=[n]$ then we will abuse
notation slightly and use $\mathcal G(n,p)$ to refer to $\mathcal
G([n],p)$. Given a property ${\cal P}$ (that is, a collection of
graphs) and a function $p = p(n) \in [0,1]$, we say that $G\sim
\mathcal G(n,p)$ has ${\cal P}$ {\em with high probability} (or whp for brevity) if the
probability that $G\in \mathcal P$ tends to $1$ as $n$ tends to
infinity.

\section{Auxiliary Lemmas}
\label{sec:auxiliarylemmas}
\subsection{Concentration inequalities}
We need to employ standard bounds on large deviations of random
variables. The following well-known lemma due to Chernoff (commonly known as the ``Chernoff bound'') provides a bound on the upper tail of the Binomial distribution
(see e.g. \cite{AlonSpencer}, \cite{JLR}).

\begin{lemma}\label{Che}
Let $X \sim \emph{\text{Bin}}(n,p)$ and let $\mu=\mathbb{E}\left[X\right]$.
Then $\Pr\left[X\ge(1+a)\mu\right]<e^{-\frac{a^2\mu}{3}}$ for any $0<a<\frac{3}{2}$.
\end{lemma}

The next lemma is a concentration  inequality for the edge exposure martingale in $\mathcal G(n,p)$ which follows easily from Theorem 7.4.3 of \cite{AlonSpencer}.

\begin{lemma}
\label{lemma:martingale}
Suppose $X$ is a random variable in the probability space $\mathcal G(n,p)$ such that $|X(G)-X(H)|\le C$ if $G$ and $H$ differ in one edge. Then
\[\Pr\left[\left|X-\EE\left[X\right]\right|>C\alpha \sqrt{n^2p}\right]\le 2e^{-\frac{\alpha^2}{4}}\]
for any positive $\alpha<2\sqrt{n^2p}$.
\end{lemma}

\subsection{Galton-Watson trees and paths}\label{s-GW}
A Galton-Watson tree is a random rooted tree, constructed
recursively from the root where each node has a random number of
children and these random numbers are independent copies of some
random variable $\xi$ taking values in $\{0,1,2,\ldots\}$. We let
$\mathcal T$ denote a (random) Galton-Watson tree. We view the
children of each node as arriving in some random order, so that
$\mathcal T$ is an ordered, or plane tree.

We consider the \textit{conditioned Galton-Watson tree} $\mathcal T_t$, which is the random tree $\mathcal T$ conditioned on having exactly $t$ vertices. In symbols, $\mathcal T_t:=(\mathcal T \mid |\mathcal T|=t)$, where, for any tree $T$, $|T|$ denotes its number of vertices.

For a rooted tree $T$, the \textit{depth} $h(v)$ of a vertex $v$ is its distance to the root (in particular the root has depth $0$). We define as usual the \textit{height} of the rooted tree $T$ by $H(T):=\max\{h(v):v\in T\}$. The following lemma which appears in \cite{ADJheight} provides essentially optimal uniform sub-Gaussian upper tail bounds on $\frac{H(\mathcal T_t)}{\sqrt{t}}$ for every offspring distribution $\xi$ with finite variance.

\begin{lemma}
\label{lemma:ADJ}
Suppose that $\mathbb E\left[\xi\right]=1$ and $0<\text{Var}\left[\xi\right]<\infty$. Then there exist constants $C,c>0$ (which may depend on $\xi$) such that
\[\Pr\left[H(\mathcal T_t)\ge h\right]\le C \exp\left(-\frac{c h^2}{t}\right)\]
for all $h\ge 0$ and $t\ge 1$.
\end{lemma}
As is well known, the distribution of the tree $\mathcal T_t$ is not changed if $\xi$ is replaced by another random variable $\xi'$ whose distribution is created from that of $\xi$  by \textit{tilting} or \textit{conjugation} (see e.g. \cite{kennedy}): if for every $k\ge 0$ we have $\Pr\left[\xi'=k\right]=c' \mu^{k}\Pr\left[\xi=k\right]$ for some $\mu>0$ and normalizing constant $c'$. Thus, we see that Lemma \ref{lemma:ADJ} remains true for $\xi\sim \text{Poisson}(\mu)$, with $\mu>0$, in which case the parameters $C,c>0$ are universal constants which do not depend on the parameter $\mu$. It is also well known (see e.g. Section $6.4$ of \cite{devroye}) that if $\xi\sim \text{Poisson}(\mu)$ then $\mathcal T_t$ is distributed as a random rooted labelled tree, that is, a tree picked uniformly from the $t^{t-1}$ trees on vertices $\{1,2,\ldots,t\}$ in which one vertex is declared to be the root. From this we obtain an estimate to be used by us later.
\begin{lemma}
\label{lemma:estimate}
Given $0\le \ell \le t$ let $p_{t,\ell}$ denote the proportion of (rooted) labeled trees on $t$ vertices which contain a path of length at least $\ell$. There exist constants $C,\varepsilon_0>0$ such that for any $\varepsilon\in (0, \varepsilon_0)$ if $\ell=\frac{C}{\varepsilon}\ln\left(\frac{1}{\varepsilon}\right)$ and $t_0=\frac{15}{\varepsilon^2}\ln \left(\frac{1}{\varepsilon}\right)$ then
\[\sum_{\ell \le t\le t_0}p_{t,\ell}\le \varepsilon^3\]
\end{lemma}

\begin{proof}[Proof of Lemma ~\ref{lemma:estimate}]
It follows from Lemma \ref{lemma:ADJ} and the considerations above that there exist constants $C',c'>0$ such that for every $t\le t_0$:
\begin{equation*}
p_{t,\ell}\le C' \exp\left(-\frac{c'\ell^2}{t}\right)\le C'\exp\left(-\frac{c'\left(\frac{C}{\varepsilon}\ln\left(\frac{1}{\varepsilon}\right)\right)^2}{\frac{15}{\varepsilon^2}\ln\left(\frac{1}{\varepsilon}\right)}\right)=C' \varepsilon^{\frac{c' C^2}{15}}.
\end{equation*}
Thus, if $C>\sqrt{\frac{90}{c'}}$ and if $\varepsilon_0$ is sufficiently small then we see that for any $\varepsilon\in (0, \varepsilon_0)$ and for $t\le t_0$ we have $p_{t,\ell}\le \varepsilon^6$. Using this we conclude that
\[\sum_{\ell \le t\le t_0}p_{t,\ell}\le \varepsilon^6\cdot t_0 =15\varepsilon^4\ln \left(\frac{1}{\varepsilon}\right)\le \varepsilon^3\,,\]
provided $\varepsilon_0$ is sufficiently small, as claimed.
\end{proof}

The next lemma concerns the sizes of Poisson Galton-Watson trees which contain long paths.

\begin{lemma}
\label{lemma:treeswithlongpaths}
For $\varepsilon>0$ let $0<\mu<1$ be such that $\mu e^{-\mu}=(1+\varepsilon)e^{-(1+\varepsilon)}$. Given $\ell\ge 1$ consider a $\text{Poisson}(\mu)$-Galton-Watson tree $\mathcal T$ and the random variable
\[T_{\ell}:=\left\{\begin{matrix}
|\mathcal T| & \text{ if } \mathcal T\text{ contains a path of length at least }\frac{\ell}{3}\\
0 & \text{ otherwise}\,,
\end{matrix}\right.\]
where $|\mathcal T|$ denotes the number of vertices of $\mathcal T$. Then there exist constants $C,\varepsilon_0>0$ such that for every $\varepsilon\in (0, \varepsilon_0)$ and for $\ell=\frac{C}{\varepsilon}\ln\left(\frac{1}{\varepsilon}\right)$ we have $\EE\left[T_{\ell}\right]\le 14\varepsilon^3$ and $\text{Var}\left[T_{\ell}\right]\le \frac{8}{\varepsilon^3}$.
\end{lemma}

\begin{proof}
We have
\begin{equation}
\label{eq:expectedvalue}
\EE\left[T_{\ell}\right]=\EE\left[\EE\left[T_{\ell}\mid |\mathcal T|\right]\right]=\sum_{t\ge 1}\Pr\left[|\mathcal T|=t\right]\cdot \EE\left[T_{\ell}\mid |\mathcal T|=t\right].
\end{equation}
It is well-known (see, e.g., Section $6.6$ of \cite{devroye}) that the size of the $\text{Poisson}(\mu)$-Galton-Watson tree $\mathcal T$ follows a $\text{Borel}(\mu)$ distribution, namely,
\begin{equation*}
\label{eq:borel}
\Pr\left[|\mathcal T|=t\right]=\frac{t^{t-1}\left(\mu e^{-\mu}\right)^{t}}{\mu\cdot  t!}.
\end{equation*}
Moreover, as discussed in the remarks that follow Lemma \ref{lemma:ADJ}, if we condition a $\text{Poisson}(\mu)$-Galton-Watson tree on it having exactly $t$ vertices then it is identically distributed to a random rooted labelled tree on $t$ vertices. Thus, it follows that $\EE\left[T_{\ell}\mid |\mathcal T|=t\right]$ is equal to $t\cdot p_{t,\frac{\ell}{3}}$, where $p_{t,\frac{\ell}{3}}$ denotes the proportion of rooted labeled trees on $t$ vertices which contain a path of length at least $\frac{\ell}{3}$. Hence, setting $t_0:=\frac{15}{\varepsilon^2}\ln\left(\frac{1}{\varepsilon}\right)$ with foresight, it follows from (\ref{eq:expectedvalue}) that
\begin{align}
\EE\left[T_{\ell}\right]&=\sum_{t\ge 1}\frac{t^{t-1}\left(\mu e^{-\mu}\right)^{t}}{\mu\cdot  t!}\cdot t\cdot p_{t,\frac{\ell}{3}}\notag\\
&\le \frac{1}{\mu}\sum_{t\ge \frac{\ell}{3}}\frac{t^t}{t!}\cdot (1+\varepsilon)^t\cdot e^{-(1+\varepsilon)t}\cdot p_{t,\frac{\ell}{3}}\notag\\
&\le 2\sum_{t\ge \frac{\ell}{3}}e^{-\frac{\varepsilon^2}{3}t}\cdot p_{t,\frac{\ell}{3}}\notag\\
&\le 2\cdot \left(\sum_{\frac{\ell}{3}\le t\le t_0} p_{t,\frac{\ell}{3}}+\sum_{t\ge t_0}e^{-\frac{\varepsilon^2}{3}t}\right)\,,
\label{eq:twosums}
\end{align}
where in the second inequality we used the facts that $\frac{t^t}{t!}\le e^{t}$, that $(1+\varepsilon)^{t}\le e^{\varepsilon t-\frac{\varepsilon^2}{3}t}$ (which holds since the first terms of the Taylor series expansion of $\ln (1+\varepsilon)$ are $\varepsilon-\frac{\varepsilon^2}{2}$) and that $\frac{1}{\mu}\le 2$ provided $\varepsilon_0$ is chosen sufficiently small. By Lemma \ref{lemma:estimate} there exist constants $C,\varepsilon_0>0$ such that the first sum in (\ref{eq:twosums}) is at most $\varepsilon^3$. Moreover, the second sum in (\ref{eq:twosums}) is
\begin{equation}
\label{eq:upperbound}
\sum_{t\ge t_0}e^{-\frac{\varepsilon^2}{3}t}= e^{-\frac{\varepsilon^2}{3} t_0}\cdot  \frac{1}{1-e^{-\frac{\varepsilon^2}{3}}}\le\varepsilon^5 \cdot \frac{6}{\varepsilon^2}=6\varepsilon^3\,,
\end{equation}
where we used the fact that $\frac{1}{1-e^{-x}}\le \frac{2}{x}$ for $x>0$ sufficiently small (which holds since the first terms of the Taylor series expansion of $e^{-x}$ are $1-x$). Thus, all in all, we conclude that there exist constants $C,\varepsilon_0>0$ such that
\[\EE\left[T_{\ell}\right]\le 2 \cdot (\varepsilon^3+6\varepsilon^3)= 14\varepsilon^3\]
as claimed. Since $|\mathcal T|\sim \text{Borel}(\mu)$ it follows that
\[\text{Var}\left[T_{\ell}\right]\le \EE\left[T^2_{\ell}\right]\le \EE\left[|\mathcal T|^2\right]=\frac{1}{(1-\mu)^3}.\]
Morever, it is straightforward to check that if $\varepsilon_0$ is chosen sufficiently small then $\mu\le 1-
\frac{\varepsilon}{2}$. Thus, we conclude that
\[\text{Var}\left[T_{\ell}\right]\le \frac{8}{\varepsilon^3}\]
as claimed.
\end{proof}

\begin{lemma}
\label{lemma:vertexdisjointpaths}
Let $P=(V,E)$ be a path of length $\ell$ and $B\subseteq E$ a set of size $|B|\le \alpha\ell$, where $\alpha\ge \frac{1}{\ell}$. Let $Q$ denote the graph obtained from $P$ by deleting all the edges in $B$. Then there exist vertex disjoint subpaths $\{Q^j\}_{j\in J}$ of $Q$ such that each $Q^j$ has length at least $\frac{1}{3\alpha}$ and the subpaths $\{Q^j\}_{i\in J}$ cover at least $\left(\frac{1}{3}-\alpha\right)\ell$ vertices of $V$.
\end{lemma}

\begin{proof}[Proof of Lemma \ref{lemma:vertexdisjointpaths}]
Since $P$ is a path, $Q$ consists of a union of vertex disjoint paths $\{Q^{j}\}_{j\in [k]}$ for some $k\le |B|+1\le \alpha \ell +1$. Denoting by $\ell_{j}$ the length of the path $Q^{j}$ for $j\in [k]$, note that
\begin{equation}
\label{eq:lengthsofpaths1}
\sum_{j\in [k]}\ell_{j}=\ell-|B|\ge (1-\alpha)\ell.
\end{equation}
Moreover, setting $J:=\{j\in [k]:\ell_{j}\ge \frac{1}{3\alpha}\}$ we see that
\begin{equation}
\label{eq:lengthsofpaths2}
\sum_{j\notin J}\ell_{j}\le k\cdot \frac{1}{3\alpha}\le \frac{1}{3}\ell+\frac{1}{3\alpha}\le \frac{2}{3}\ell.
\end{equation}
Putting (\ref{eq:lengthsofpaths1}) and (\ref{eq:lengthsofpaths2}) together we get that
\[\sum_{j\in J}\ell_{j}\ge \left(\frac{1}{3}-\alpha\right)\ell.\]
Thus, it follows that the paths $\{Q^j\}_{j\in J}$ satisfy the desired conditions.
\end{proof}

\subsection{Properties of random graphs}
The next lemma provides bounds on the sizes of the largest and second largest connected components of $G\sim \mathcal G(n,p)$ as well as the size of its $2$-core when $p=\frac{1+\varepsilon}{n}$, where $\varepsilon>0$ is a small constant. This lemma is a simple consequence of Theorem 5.4 of \cite{JLR} and Theorem 3 of \cite{2-core}.

\begin{lemma}
\label{lemma:supercriticalregime}
Let $p=\frac{1+\varepsilon}{n}$ where $\varepsilon>0$ is a constant. Then there exists a constant $\varepsilon_0>0$ such
that for every $\varepsilon\in (0, \varepsilon_0)$ the following holds whp for $G\sim \mathcal G(n,p)$:
\begin{enumerate}[$(a)$]
\item the largest connected component of $G$ has between $\varepsilon n$ and $3\varepsilon n$ vertices.

\item the second largest connected component of $G$ has at most $\frac{20}{\varepsilon^2}\ln n$ vertices.

\item the $2$-core of the largest connected component of $G$ has at most $2\varepsilon^2n$ vertices.
\end{enumerate}
\end{lemma}

In \cite{DLP}, Ding, Lubetzky and Peres established a complete characterization of the structure of the giant component $\mathcal C_1$ of $G\sim \mathcal G(n,p)$ in the strictly supercritical regime ($p=\frac{1+\varepsilon}{n}$ with $\varepsilon>0$ constant). This was achieved by offering a tractable contiguous model $\tilde{\mathcal{C}}_1$, i.e. a model such that every graph property that is satisfied by $\tilde{\mathcal C}_1$ whp is also satisfied by $\mathcal C_1$ whp. In their model, $\tilde{\mathcal C}_1$ consists of a $2$-core $\tilde{\mathcal C}^{(2)}_1$ where one attaches to each vertex  of $\tilde{\mathcal C}^{(2)}_1$ one independent $\text{Poisson}(\mu)$-Galton-Watson tree (where $0<\mu<1$ is such that $\mu e^{-\mu}=(1+\varepsilon)e^{-(1+\varepsilon)}$). In light of this, any graph property that is satisfied whp by the disjoint union of $|\tilde{\mathcal C}^{(2)}_1|$ independent $\text{Poisson}(\mu)$-Galton-Watson trees must also be satisfied whp by $\mathcal C_1\setminus \mathcal C_1^{(2)}$, the graph obtained from the giant component $\mathcal C_1$ by removing the edges of its $2$-core $\mathcal C_1^{(2)}$. As one would expect, the random variable $|\tilde{\mathcal C}^{(2)}_1|$ is tightly concentrated around its expectation, which agrees with the expected size of the $2$-core $\mathcal C^{(2)}_1$ of $\mathcal C_1$. By $(c)$ of Lemma \ref{lemma:supercriticalregime} this at most $2\varepsilon^2n$. The next technical lemma which will be useful in the proof of Theorem~\ref{thm:notmanyvertexdisjointpaths} follows from the considerations above.

\begin{lemma}
\label{lemma:DLP}
Let $\mathcal C_1$ denote the largest connected component of $G\sim \mathcal G(n,p)$ for $p=\frac{1+\varepsilon}{n}$, where $\varepsilon>0$ is fixed, let $\mathcal C^{(2)}_1$ denote its $2$-core and let $\mathcal C_1\setminus \mathcal C^{(2)}_1$ denote the graph obtained from $\mathcal C_1$ by removing the edges in $\mathcal C^{(2)}_1$. Let $0<\mu<1$ be such that $\mu e^{-\mu}=(1+\varepsilon)e^{-(1+\varepsilon)}$ and consider $2\varepsilon^2n$ independent $\text{Poisson}(\mu)$-Galton-Watson trees $\mathcal T_1,\ldots, \mathcal T_{2\varepsilon^2n}$. Then, for every $\ell$ and $m$ (which might depend on $n$) if whp the disjoint union of $\mathcal T_1,\ldots, \mathcal T_{2\varepsilon^2n}$ does not contain a set of vertex disjoint paths of length at least $\ell$ covering at least $m$ vertices then the same holds whp for $\mathcal C_1\setminus \mathcal C^{(2)}_{1}$.
\end{lemma}

\section{Proof of Theorem \ref{thm:noalgorithm}}
\label{sec:proofthm1}

We start this section by repeating the statement of Theorem \ref{thm:noalgorithm} for the reader's convenience.
\begin{theorem}
\label{thm:noalgorithm} There exists an absolute constant $C>0$ such
that the following holds. For every constant $q\in (0,1)$ there
exist $n_0,\varepsilon_0>0$ such that for every fixed
$\varepsilon\in (0,\varepsilon_0)$ and any $n\ge n_0$ there is no
adaptive algorithm which reveals a path of length $\ell\ge
\frac{3C}{\varepsilon}\ln\left(\frac{1}{\varepsilon}\right)$ with
probability at least $q$ in $G\sim\mathcal G\left(n,p\right)$, where
$p=\frac{1+\varepsilon}{n}$, by querying at most
$\frac{q\ell}{8640Cp\varepsilon\ln\left(\frac{1}{\varepsilon}\right)}$
pairs of vertices.
\end{theorem}

\begin{proof}[Proof of Theorem \ref{thm:noalgorithm}]
Suppose $\text{Alg}$ is an adaptive algorithm which with probability
at least $q$ finds a path of length $\ell$ in $G\sim\mathcal
G\left(n,p\right)$, where $p=\frac{1+\varepsilon}{n}$, after querying at most
$\frac{q\ell}{8640Cp\varepsilon\ln\left(\frac{1}{\varepsilon}\right)}$
pairs of vertices. We consider implicitly that $\text{Alg}$ takes an \textit{ordered} vertex set as part of its input.  We shall assume henceforth that $n,C>0$ are
sufficiently large and $\varepsilon>0$ is sufficiently small in
order to obtain a contradiction. Note that, by restricting
$\text{Alg}$ to a set of $n$ vertices, we get an algorithm which for
any $n'\ge n$ with probability at least $q$ finds in $G'\sim
\mathcal G(n',p)$ a path of length $\ell$ after querying at most
$\frac{q\ell}{8640Cp\varepsilon\ln\left(\frac{1}{\varepsilon}\right)}$
pairs of vertices. We shall abuse notation slightly and call Alg to all these algorithms.

Define
$n':=\left(1+\frac{720\varepsilon^2}{q}\right)n$, $V_0:=[n']$,
$I_{0}:=\emptyset$ and $s:=\frac{720\varepsilon^2 n}{q(\ell+1)}$.
For $i=1,\ldots,s$ do the following:
\begin{itemize}
\item Apply $\text{Alg}$ to $G_{i-1}\sim\mathcal G\left(V_{i-1},p\right)$, where the vertices in $V_{i-1}$ are permuted according to a permutation $\pi_{i}\in S_{V_{i-1}}$ chosen uniformly at random. Let $L_{i}$ be the graph of all pairs of vertices queried and let $K_{i}\subseteq L_{i}$ be the graph of edges present. By the algorithm we know that $L_{i}$ has at most $\frac{q\ell}{8640Cp\varepsilon\ln\left(\frac{1}{\varepsilon}\right)}$ edges. If $K_i$ contains a path of length $\ell$ then let $P_i$ be one such path, define $V_{i}:=V_{i-1}\setminus V(P_i)$ and set $I_{i}:=I_{i-1}\cup \{i\}$. Otherwise, set $V_{i}:=V_{i-1}$ and $I_{i}:=I_{i-1}$.
\end{itemize}

Observe that $|V_{s}|\ge n'-(\ell+1) s=
\left(1+\frac{720\varepsilon^2}{q}\right)n-\frac{720\varepsilon^2}{q}
n= n$ and so we can indeed apply $\text{Alg}$ to $V_{i-1}$ for any
$i\in [s]$. We define a random graph $H$ with vertex set $V_0$ in
the following way. For every pair of vertices $\{u,v\}\subseteq V_0$
if $\{u,v\}\in E(L_{i})$ for some $i\in [s]$ then  let $i_0$ be the
smallest such index and set $\{u,v\}$ as an edge of $H$ if and only
if $\{u,v\}\in E(K_{i_{0}})$. Consider all the other pairs
$\{u,v\}\subseteq V_0$ as non-edges of $H$. From the procedure above
it follows that for every $\{u,v\}\subseteq V_0$ we have
independently that
\[\Pr\left[\{u,v\}\in E(H)\right]\le p=\frac{1+\varepsilon}{n}=\frac{1+\varepsilon}{n'}\cdot \frac{n'}{n}=\frac{(1+\varepsilon)\left(1+\frac{720\varepsilon^2}{q}\right)}{n'}\le
\frac{1+2\varepsilon}{n'}\,,\] provided $\varepsilon\le
\frac{q}{1440}$. Thus, the graph $H$ can be viewed as a subgraph of
a graph sampled from $\mathcal
G\left(n',\frac{1+2\varepsilon}{n'}\right)$. In particular, if with
probability at least $\frac{q^2}{4}$ the graph $H$ contains a set of
vertex disjoint paths of length at least
$\frac{C}{\varepsilon}\ln\left(\frac{1}{\varepsilon}\right)$ which
cover at least $52\varepsilon^2n'$ vertices then the same must also
hold with probability at least $\frac{q^2}{4}$ in $\mathcal
G\left(n',\frac{1+2\varepsilon}{n'}\right)$. However, this would
contradict Theorem \ref{thm:notmanyvertexdisjointpaths} and so it
suffices to prove the following claim:
\begin{claim}
With probability at least $\frac{q^2}{4}$ the graph $H$ contains a
set of vertex disjoint paths of length at least
$\frac{C}{\varepsilon}\ln\left(\frac{1}{\varepsilon}\right)$ which
cover at least $52\varepsilon^2n'$ vertices of $V_0$.
\end{claim}
Define for each $i\in I_{s}$ the graph $H_{i}$ with vertex set $V_{i-1}$ and edge set $\left(\bigcup_{j=1}^{i-1}E(L_j)\right)\cap \binom{V_{i-1}}{2}$ and note that
\begin{equation}
\label{eq:edges}
|E(H_i)|\le s\cdot \frac{q\ell}{8640Cp\varepsilon\ln\left(\frac{1}{\varepsilon}\right)}\le\frac{\varepsilon n^2}{12C\ln \left(\frac{1}{\varepsilon}\right)(1+\varepsilon)}\le  \frac{\varepsilon}{6C\ln \left(\frac{1}{\varepsilon}\right)}\cdot \binom{|V_{i-1}|}{2}.
\end{equation}
Observe that for each $i\in I_{s}$ the set $V_{i-1}\setminus V_{i}$
consists of the vertex set of a path $P_i$ in the graph $K_i$. For
each such $i$ set $B_{i}:=E(P_{i})\cap E(H_{i})$ and let $Q_{i}$
denote the graph obtained from $P_{i}$ by deleting all the edges in
$B_{i}$. Note crucially that  $E(Q_{i})\subseteq E(H)$ and that the
graphs $\{Q_{i}\}_{i\in I_{s}}$ are vertex disjoint.

Consider now the set $I:=\left\{i\in I_{s}:|B_i|\le
\frac{\varepsilon}{3C\ln\left(\frac{1}{\varepsilon}\right)}
\ell\right\}$. By Lemma \ref{lemma:vertexdisjointpaths} it follows
that for any $i\in I$ there exist vertex disjoint subpaths
$\{Q^j_{i}\}_{j\in J_{i}}$ of $Q_{i}$ each of length at least
$\frac{C}{\varepsilon}\ln\left(\frac{1}{\varepsilon}\right)$ which
cover at least
$\left(\frac{1}{3}-\frac{\varepsilon}{3C\ln\left(\frac{1}{\varepsilon}\right)}\right)\ell\ge\frac{1}{4}(\ell+1)$
vertices of $V(Q_{i})$. Thus, if $|I|\ge \frac{1}{3}sq$ then
$\{Q^{j}_{i}\}_{i\in I,j\in J_{i}}$ forms a collection of vertex
disjoint paths in $H$ of length at least
$\frac{C}{\varepsilon}\ln\left(\frac{1}{\varepsilon}\right)$ which
cover at least $\frac{1}{4}(\ell+1) \cdot
\frac{1}{3}sq=60\varepsilon^2n\ge 52\varepsilon^2n'$ vertices of
$V_{0}$. It suffices to show then that with probability at least
$\frac{q^2}{4}$ we have $|I|\ge \frac{1}{3}sq$.

Let $I':=[s]\setminus I$ and note that for every $i\in [s]$ we have
\begin{equation}
\label{eq:probability}
\Pr\left[i\in I'\right]=\Pr\left[i\notin I_s\right]+\Pr\left[i\in I'\mid i\in I_s\right]\cdot \Pr\left[i\in I_s\right].
\end{equation}
It is clear from the procedure above that for each $i\in [s]$ we have $\Pr\left[i\in I_{s}\right]\ge q$. Note also crucially that, provided $i\in I_s$, the path $P_i$ is a randomly mapped path of length $\ell$ on the vertex set $V_{i-1}$. Indeed, this happens because before the $i$-th application of Alg we permuted the vertices of $V_{i-1}$ according to a permutation $\pi_i\in S_{V_{i-1}}$ chosen uniformly at random. Thus, by conditioning on the event that $i\in I_{s}$, on any possible graph $H_i$ satisfying (\ref{eq:edges}) and on the path $\pi_{i}^{-1}(P_i)$, we have for any $e\in E\left(\pi_{i}^{-1}(P_i)\right)$:
\[\Pr\left[\pi_{i}(e)\in E(H_{i})\right]\le \frac{\varepsilon}{6C\ln \left(\frac{1}{\varepsilon}\right)}\,,\]
and so, by linearity of expectation it follows that:
\[\EE\left[|E(P_{i})\cap E(H_{i})|\right]\le \frac{\varepsilon}{6C\ln \left(\frac{1}{\varepsilon}\right)}\ell.\]
Thus, by Markov's inequality (see, e.g., \cite{AlonSpencer}) we get that
\[\Pr\left[i\in I'\mid i\in I_{s}\right]\le \frac{1}{2}\,,\]
and so by equation (\ref{eq:probability}) we see that for any $i\in [s]$ we have $\Pr\left[i\in I'\right]\le 1-\frac{1}{2}\Pr[i\in I_s]\le 1-\frac{q}{2}$. It follows then by linearity of expectation that $\EE\left[|I'|\right]\le s\left(1-\frac{q}{2}\right)$. Hence, again by Markov's inequality we conclude that
\[\Pr\left[|I'|\ge  \frac{s}{1+\frac{q}{2}}\right]\le 1-\frac{q^2}{4}\;\;\text{, which implies } \;\; \frac{q^2}{4}\le \Pr\left[|I|\ge \frac{sq}{2+q}\right]\le \Pr\left[|I|\ge \frac{sq}{3}\right].\]
This completes the proof.
\end{proof}

\section{Proof of Theorem \ref{thm:notmanyvertexdisjointpaths}}
\label{sec:proofthm2}
\begin{theorem}
\label{thm:notmanyvertexdisjointpaths} There exist constants
$C,\varepsilon_0>0$ such that for every fixed $\varepsilon\in(0,
\varepsilon_0)$ we have whp that $G\sim \mathcal
G\left(n,\frac{1+\varepsilon}{n}\right)$ does not contain a set of
vertex disjoint paths of lengths at least $\frac{C}{\varepsilon}\ln
\left(\frac{1}{\varepsilon}\right)$ whose union covers at least
$13\varepsilon^2n$ vertices.
\end{theorem}

\begin{proof}[Proof of Theorem \ref{thm:notmanyvertexdisjointpaths}]
Let $G\sim \mathcal G(n,p)$ where $p=\frac{1+\varepsilon}{n}$. Let $\mathcal C_1$ denote the largest connected component of $G$, let $\mathcal C^{(2)}_1$ denote the $2$-core of $\mathcal C_1$ and let $\mathcal C_1\setminus \mathcal C^{(2)}_1$ denote the graph obtained from $\mathcal C_1$ by deleting the edges in $\mathcal C^{(2)}_1$. For $\ell\ge 1$ consider the following random variables:
\begin{itemize}
\item $X_{\ell}$ = number of vertices which belong to connected components of $G$ of size at most $\frac{20}{\varepsilon^2}\ln n$ containing a path of length at least $\ell$.

\item $Y_{\ell}$ = maximum number of vertices covered by vertex disjoint paths of length at least $\ell$ in $\mathcal C_1$.

\item $Z_{\ell}$ = maximum number of vertices covered by vertex disjoint paths of length at least $\frac{\ell}{3}$ in $\mathcal C_1\setminus \mathcal C^{(2)}_1$.
\end{itemize}

By $(b)$ of Lemma \ref{lemma:supercriticalregime} it follows that whp $X_{\ell}+Y_{\ell}$ is an upper bound on the maximum number of vertices of $G$ covered by  vertex disjoint paths of length at least $\ell$. Note that we may assume that all the paths considered have size at most $2\ell$ by splitting larger paths into several paths of length at least $\ell$. Moreover, if $P$ is a path of length at least $\ell$ in $\mathcal C_1$ then, since $\mathcal C_1\setminus \mathcal C^{(2)}_1$ consists of a disjoint union of trees, there must exist a subpath $P'$ of the path $P$ with at least $\frac{|P|}{3}\ge \frac{\ell}{3}$ vertices which lies in $\mathcal C^{(2)}_1$ or in $\mathcal C_1\setminus \mathcal C^{(2)}_1$. Since $|P|\le 6|P'|$ it follows that $Y_{\ell}\le 6|\mathcal C^{(2)}_1|+6Z_{\ell}$.

By (c) of Lemma \ref{lemma:supercriticalregime} we know that whp $|\mathcal C^{(2)}_1|\le 2\varepsilon^2n$, provided $\varepsilon_0$ is chosen small enough. It suffices then to show that there exist constants $C,\varepsilon_0>0$ such that for every fixed $\varepsilon\in (0, \varepsilon_0)$ and for $\ell:=\frac{C}{\varepsilon}\ln \left(\frac{1}{\varepsilon}\right)$ we have whp that
\[X_{\ell}<\varepsilon^{3}n\;\;\;\text{ and }\;\;\; Z_{\ell}< 29\varepsilon^5n.\]
since in that case we have whp that the maximum number of vertices of $G$ covered by vertex disjoint paths of length at least $\ell$ is at most
\[X_{\ell}+Y_{\ell}\le X_{\ell}+6|\mathcal C^{(2)}_1|+6Z_{\ell} < \varepsilon^3n+6\cdot 2\varepsilon^2n+6\cdot 29\varepsilon^5n\le 13\varepsilon^2n.\]
provided $\varepsilon_0$ is chosen sufficiently small. Lemmas \ref{lemma:notmanypathsinsmallcomponents} and \ref{lemma:outside2core} complete the proof.\end{proof}

\begin{lemma}
\label{lemma:notmanypathsinsmallcomponents}
There exist constants $C,\varepsilon_0>0$ such that for every fixed $\varepsilon\in (0, \varepsilon_0)$ and for $\ell:=\frac{C}{\varepsilon}\ln \left(\frac{1}{\varepsilon}\right)$ we have $X_{\ell}<\varepsilon^{3}n$ whp.
\end{lemma}

\begin{proof}[Proof of Lemma \ref{lemma:notmanypathsinsmallcomponents}]
Given a set $S\subseteq [n]$ of size $t$, let $\mathcal S_{\ell}(S)$ (resp. $\mathcal T_{\ell}(S)$) denote the set of possible connected graphs (resp. spanning trees) on the vertex set $S$ which contain a path of length at least $\ell$. Let $X_{S}$ denote the indicator random variable of the event that $G[S]\in \mathcal S_{\ell}(S)$ and that there are no edges in $G$ between $S$ and $[n]\setminus S$. Note that $G[S]\in \mathcal S_{\ell}(S)$ if and only if there exists $T\in \mathcal T_{\ell}(S)$ such that $T\subseteq G[S]$. Thus, by the union bound we have
\begin{equation}
\label{eq:expectationX_C}
\EE\left[X_S\right]\le |\mathcal T_{\ell}(S)|\cdot p^{t-1}\cdot \left(1-p\right)^{t(n-t)}
\end{equation}
where the first term accounts for taking a union bound over all $T\in \mathcal T_{\ell}(S)$, the second term accounts for the probability that the edges in $T$ are present in $G[S]$ and the last term accounts for the probability that none of the edges between $S$ and $[n]\setminus S$ are present in $G$. Note that $|\mathcal T_{\ell}(S)|$ does not depend on the set $S$ and is equal to the number of labeled trees on $t$ vertices which contain a path of length at least $\ell$. More specifically, if $p_{t,\ell}$ denotes the proportion of labeled trees on $t$ vertices which contain a path of length at least $\ell$, then $|\mathcal T_{\ell}(S)|=p_{t,\ell}\cdot t^{t-2}$. Observe now that the random variable $X_{\ell}$ satisfies the following:
\[X_{\ell}\le \sum_{t=\ell}^{\frac{20}{\varepsilon^2}\ln n}\sum_{S\in \binom{[n]}{t}}t \cdot X_{S}.\]
We claim that for $\ell:=\frac{C}{\varepsilon}\ln \left(\frac{1}{\varepsilon}\right)$, where $C>0$ is a large constant, and for some constant $\varepsilon_0>0$, if $\varepsilon\in (0,\varepsilon_0)$ is fixed then $\Pr\left[X_{\ell}\ge \varepsilon^3n\right]= o(1)$. To prove this claim we start by estimating $\EE[X_{\ell}]$. Setting $t_0:=\frac{15}{\varepsilon^2}\ln\left(\frac{1}{\varepsilon}\right)$, we have by the linearity of expectation and by (\ref{eq:expectationX_C}) that if $\varepsilon_0$ is sufficiently small then:
\begin{align}
\EE[X_{\ell}]&\le \sum_{t=\ell}^{\frac{20}{\varepsilon^2}\ln n}t\cdot \binom{n}{t}\cdot p_{t,\ell}\cdot t^{t-2}\cdot p^{t -1}\cdot \left(1-p\right)^{t(n-t)}\notag\\
&\le \sum_{t=\ell}^{\frac{20}{\varepsilon^2}\ln n} t\cdot \left(\frac{en}{t}\right)^{t}\cdot p_{t,\ell}\cdot t^{t-2}\cdot \left(\frac{1+\varepsilon}{n}\right)^{t-1}\left(1-\frac{1+\varepsilon}{n}\right)^{t(n-t)}\notag\\
&\le \sum_{t=\ell}^{\frac{20}{\varepsilon^2}\ln n} e^t\cdot t^{-1} \cdot n \cdot p_{t,\ell}\cdot \frac{e^{\varepsilon t-\frac{\varepsilon^2}{3}t}}{1+\varepsilon}\cdot e^{-(1+\varepsilon)t+\frac{(1+\varepsilon)t^2}{n}}\notag\\
&\le \frac{(1+o(1))n}{\ell(1+\varepsilon)} \cdot \sum_{t\ge\ell}p_{t,\ell}\cdot e^{-\frac{\varepsilon^2}{3}t}\notag\\
&\le \frac{n}{14}\cdot \left( \sum_{\ell \le t\le t_0}p_{t,\ell}+\sum_{t\ge t_0}e^{-\frac{\varepsilon^2}{3}t}\right)\label{eq:twosums1}
\end{align}
where in the third inequality we used the fact that $(1+\varepsilon)^{t}\le e^{\varepsilon t-\frac{\varepsilon^2}{3}t}$ for sufficiently small $\varepsilon>0$. By Lemma \ref{lemma:estimate} there exist constants $C,\varepsilon_0>0$ such that the first sum in (\ref{eq:twosums1}) is at most $\varepsilon^3$. Moreover, by (\ref{eq:upperbound}) the second sum in (\ref{eq:twosums1}) is at most $6\varepsilon^3$. Thus, all in all, we conclude that there exist constants $C,\varepsilon_0>0$ such that
\[\EE\left[X_{\ell}\right]\le \frac{n}{14} \cdot (\varepsilon^3+6\varepsilon^3)= \frac{\varepsilon^3n}{2}.\]

Note that if $G$ and $H$ differ in precisely one edge then $|X_{\ell}(G)-X_{\ell}(H)|\le \frac{40}{\varepsilon^2}\ln n$ because one edge affects at most two connected components of size at most $\frac{20}{\varepsilon^2}\ln n$. Thus, by Lemma \ref{lemma:martingale} it follows that
\[\Pr\left[X_{\ell}>\varepsilon^3n\right]\le \Pr\left[|X_{\ell}-\EE[X_{\ell}]|>\frac{\varepsilon^3n}{2}\right]\le e^{-\Omega\left(\frac{n}{(\ln n)^2}\right)}=o(1).\]
\end{proof}

\noindent{\bf Remark.} An alernative approach to the proof of Lemma \ref{lemma:notmanypathsinsmallcomponents} would be to invoke the so called symmetry rule (see, e.g., Chapter 5.6 of \cite{JLR}), postulating that in the supercritical regime $p=\frac{1+\varepsilon}{n}$, the subgraph of $G\sim \mathcal G(n,p)$ outside the giant component behaves typically as a random graph with subcritical edge probability. One can then estimate the likely contribution of paths of length at least $\ell=\frac{C}{\varepsilon}\ln\left(\frac{1}{\varepsilon}\right)$ coming from the small components to the total volume of vertex disjoint paths of length at least $\ell$ and to show it to be $O(\varepsilon^2n)$ whp, using a direct first moment argument. Since we still need to treat the paths residing in the giant component outside the 2-core (the random variable $Z_{\ell}$), we chose to adopt a unified approach using the machinery of Galton-Watson trees developed in Section \ref{s-GW}, and to apply it here as well.

\begin{lemma}
\label{lemma:outside2core}
There exist constants $C,\varepsilon_0>0$ such that for every fixed $\varepsilon\in (0, \varepsilon_0)$ and for $\ell:=\frac{C}{\varepsilon}\ln \left(\frac{1}{\varepsilon}\right)$ we have $Z_{\ell}<29\varepsilon^{5}n$ whp.
\end{lemma}

\begin{proof}[Proof of Lemma \ref{lemma:outside2core}]
Recall that $Z_{\ell}$ counts the maximum number of vertices covered by vertex disjoint paths of length at least $\frac{\ell}{3}$ in $\mathcal C_1\setminus \mathcal C^{(2)}_1$. Let $0<\mu<1$ be such that $\mu e^{-\mu}=(1+\varepsilon)e^{-(1+\varepsilon)}$ and consider $2\varepsilon^2n$ independent $\text{Poisson}(\mu)$-Galton-Watson trees $\mathcal T_1,\ldots,\mathcal T_{2\varepsilon^2n}$. By Lemma \ref{lemma:DLP} it suffices for our purposes to show that whp the maximum number of vertices covered by vertex disjoint paths of length at least $\frac{\ell}{3}$ in the disjoint union of $\mathcal T_1,\ldots, \mathcal T_{2\varepsilon^2n}$ is less than $29\varepsilon^{5}n$, for appropriate $C,\varepsilon_0>0$.

For each $1\le i\le 2\varepsilon^2n$ consider the following random variable:
\[T_{i,\ell}:=\left\{\begin{matrix}
|\mathcal T_i| & \text{ if } \mathcal T_i\text{ contains a path of length at least }\frac{\ell}{3}\\
0 & \text{ otherwise}
\end{matrix}\right.\]
and set $T_{\ell}=\sum_{i=1}^{2\varepsilon^2n}T_{i,\ell}$. Clearly $T_{\ell}$ is an upperbound on the maximum number of vertices covered by vertex disjoint paths of length at least $\frac{\ell}{3}$ in in the disjoint union of $\mathcal T_1,\ldots, \mathcal T_{2\varepsilon^2n}$. To finish the proof, we show that whp $T_{\ell}<29\varepsilon^{5}n$, provided $C,\varepsilon_0>0$ are chosen appropriately.

By Lemma \ref{lemma:treeswithlongpaths} we know that there exist constants $C,\varepsilon_0>0$ such that for every $\varepsilon\in (0, \varepsilon_0)$ and for $\ell=\frac{C}{\varepsilon}\ln\left(\frac{1}{\varepsilon}\right)$ we have $\EE\left[T_{i,\ell}\right]\le 14\varepsilon^3$ and $\text{Var}\left[T_{i,\ell}\right]\le \frac{8}{\varepsilon^3}$. Thus, since the random variables $T_{i,\ell}$ are independent, we have that
\[\EE\left[T_{\ell}\right]\le 14\varepsilon^3\cdot 2\varepsilon^2n= 28\varepsilon^5n\;\;\; \text{ and }\;\;\; \text{Var}\left[T_{\ell}\right]\le \frac{8}{\varepsilon^3}\cdot 2\varepsilon^2n= \frac{16n}{\varepsilon}.\]
Thus, by Chebyshev's Inequality (see, e.g., \cite{AlonSpencer}) we conclude that
\[\Pr\left[T_{\ell}\ge 29\varepsilon^5n\right]\le \Pr\left[|T_{\ell}-\EE\left[T_{\ell}\right]|\ge \varepsilon^5n\right]\le \frac{\text{Var}\left[T_{\ell}\right]}{\varepsilon^{10}n^2}\le \frac{16}{\varepsilon^{11}n}=o(1).\]\end{proof}

\section{Concluding remarks}
\label{sec:concludingremarks} We have shown that in order
to find a path of length $\ell=\Omega\left(\frac{\log\left(\frac{1}{\varepsilon}\right)}{\varepsilon}\right)$ in $G\sim\mathcal G\left(n,p\right)$
with at least some constant probability, where
$p=\frac{1+\varepsilon}{n}$ with $\varepsilon>0$ fixed, one needs to
query at least $\Omega\left(\frac{\ell}{p\varepsilon
\log\left(\frac{1}{\varepsilon}\right)}\right)$ pairs of vertices.
This is close to best possible since a randomised depth first search
algorithm from \cite{ks13} finds whp a path of length $\ell$ after
querying at most $O\left(\frac{\ell}{p\varepsilon}\right)$ pairs of
vertices. A natural question, which remains open, is to close the
gap between these bounds. We believe that every adaptive algorithm
which reveals whp a path of length $\ell$ in $G\sim\mathcal G(n,p)$,
where $p=\frac{1+\varepsilon}{n}$ with $\varepsilon>0$ fixed, has to
query $\Omega\left(\frac{\ell}{p\varepsilon}\right)$ pairs of
vertices.

Recall that, to prove our main result, in
Theorem~\ref{thm:notmanyvertexdisjointpaths} we bounded the total
number of vertices covered by vertex disjoint paths of size at least
$\Omega\left(\frac{1}{\varepsilon}\log\left(\frac{1}{\varepsilon}\right)\right)$
in a typical graph sampled from $\mathcal G(n,p)$, $p=\frac{1+\varepsilon}{n}$, by
$O\left(\varepsilon^2 n\right)$. Since a graph $G\sim\mathcal
G(n,p)$ contains whp a path of length $\Theta(\varepsilon^2n)$ (see
e.g. \cite{JLR}), this is best possible up to a
multiplicative constant. If one can show that a similar statement
holds for paths of length $\Omega\left(\frac{1}{\varepsilon}\right)$
then one can modify our proof to obtain a
$\Omega\left(\frac{\ell}{p\varepsilon}\right)$ bound in
Theorem~\ref{thm:noalgorithm}.

In the proof of Theorem~\ref{thm:notmanyvertexdisjointpaths} we needed to bound the number of vertices covered by vertex disjoint paths of a prescribed length $\ell$ in a random tree of fixed size $t$ (Lemma~\ref{lemma:treeswithlongpaths}). Our estimate was a bit wasteful because for trees which contained a path of length $\ell$ we used their total number of vertices $t$ instead of the number of vertices covered by vertex disjoint paths of length $\ell$, which is most likely significantly smaller. A way to fix this is to obtain good bounds for the following question:
\begin{ques}
Given $a=a(t)\in \NN$ and $b=b(t)\in \NN$ what is the probability that a random tree on $t$ vertices contains  $b$ vertex disjoint paths, each of length at least $a$?
\end{ques}
Note that, since the diameter of a random tree on $t$ vertices is whp $\Theta(\sqrt{t})$ (see e.g. \cite{ADJheight}), the only interesting regime is when $ab\ge C\sqrt{t}$ for some constant $C>0$. Moreover, by splitting paths of length larger than $2a$ into smaller subpaths of length at least $a$, we may consider only paths of length between $a$ and $2a$.

One possible approach to this problem would be through a nice argument of Joyal (\cite{Joyal}, see also \cite{ProofsFromTheBook}). It shows that a random tree $\mathcal T$ on $t$ vertices can be obtained from a random map $f:[t]\rightarrow [t]$ as follows. First we create the directed graph $D$ (possibly with loops) on vertex set $[t]$ with edges $i\rightarrow f(i)$ for each $i\in [t]$. Then we look at a maximal set of vertices $M=\{i_1,\ldots,i_m\}\subseteq [t]$ such that $f|_{M}$ is a permutation. We remove the directed edges inside $M$ and replace them by the path $f(i_1)\rightarrow f(i_2)\rightarrow\ldots \rightarrow f(i_m)$ (where $i_1<i_2<\ldots<i_m$). By ignoring the orientations of the edges we obtain the desired tree $\mathcal T$. Note that, since the vertices in $M$ form a path in $\mathcal T$, we must have $|M|=O(\sqrt{t})$ whp. Moreover, if we have a path $P$ in $\mathcal T$ then a moment's thought reveals that either $P$ has at least $\frac{|V(P)|}{3}$ vertices in $M$ or there are $\frac{|V(P)|}{3}$ vertices of $P$ which form a directed path in $D$. Thus, it follows that if we have a collection of $b$ vertex disjoint paths in $\mathcal T$ each of length between $a$ and $2a$ then $D$ contains a collection of vertex disjoint directed paths each of length between $\frac{a+1}{3}$ and $2a$ covering at least $\frac{(a+1)b}{3}-|M|$ vertices. Since $|M|=O(\sqrt{t})$ whp and since we are interested only in the case when $ab\ge C\sqrt{t}$ for some large constant $C>0$, it follows that in that case we have, say, at least $\frac{b}{10}$ such paths. Thus, up to changing $a$ and $b$ by constant multiplicative factors, it is enough to estimate the probability   that the directed graph $D$ obtained from a random map $f:[t]\rightarrow [t]$ contains at least $b$ vertex disjoint directed paths, each of length (at least) $a$.

We can give a simple upper bound on this probability by taking the union bound over all collections of $b$ vertex disjoint directed paths of length $a$. This shows that the probability that we want to estimate is at most
\[\frac{t!}{(t-(a+1)b)!b!}\left(\frac{1}{t}\right)^{ab}=\frac{t^b}{b!}\prod_{i=1}^{(a+1)b-1}\left(1-\frac{i}{t}\right)\le e^{b+b\ln\left(t/b\right)-\binom{(a+1)b}{2}/t}.\]
Unfortunately, this upper bound is not strong enough to allow us to prove Theorem~\ref{thm:notmanyvertexdisjointpaths} for paths of length at least $\Omega\left(\frac{1}{\varepsilon}\right)$ because when $b$ is roughly a constant and $a$ is close to $\sqrt{t}$ the positive term $b\ln\left(t/b\right)$ in the exponent is much larger than the negative term $\binom{(a+1)b}{2}/t$. Thus, it would be nice to obtain tighter bounds for the probability in question.

\medskip

\textbf{Acknowledgement.} We would like to thank the anonymous referees for carefully reading our paper as well as for their helpful comments. Parts of this work were carried out when the third author visited the School of Mathematical Sciences of Tel Aviv University, Israel. We would like to thank this institution for its hospitality and for creating a stimulating research environment.

\end{document}